\documentclass[11pt,leqno]{amsart}
\usepackage{amsmath,amsthm,amssymb}
\usepackage{tabularx}
\usepackage{comment}
\usepackage{url} 
\usepackage{multicol}

\newcommand{\GF}{{\mathbb F}}
\newcommand{\FF}{{\mathbb F}}

\newcommand{\R}{{\mathbb R}}
\newcommand{\RR}{{\mathbb R}}

\newcommand{\NN}{{\mathbb N}}

\newcommand{\PAut}{{\rm PAut}}

\newcommand{\wt}{{\rm wt}}
\newcommand{\supp}{{\rm supp}}

\DeclareMathOperator{\Harm}{Harm}

\usepackage{color} 
\usepackage{ulem}
\usepackage{blkarray}

\newtheorem{thm}{Theorem}[section]
\newtheorem{lem}[thm]{Lemma}
\newtheorem{cor}[thm]{Corollary}
\newtheorem{prop}[thm]{Proposition}

\theoremstyle{definition}
\newtheorem{df}[thm]{Definition}
\newtheorem{rem}[thm]{Remark}
\newtheorem{ex}[thm]{Example}

\numberwithin{equation}{section}

\title[$3$-designs in the extended quadratic residue code]
{Infinite series of $3$-designs in the extended quadratic residue code}

\author[Madoka Awada]{Madoka Awada*}
\thanks{*Corresponding author}
\address{		Graduate School of Science and Engineering, 
		Waseda University, 
		Tokyo 169--8555, Japan
}
\email{madoka-awada@fuji.waseda.jp} 

\keywords{extended quadratic residue code, 
$t$-design, 
Jacobi polynomial, harmonic weight enumerator
}

\subjclass[2010]{Primary 94B05; Secondary 05B05}

\begin{document}
\begin{abstract}
In this paper, 
we show infinite series of $3$-designs in the extended quadratic residue codes over $\FF_{r^2}$ for a prime $r$.
\color{black}
\end{abstract}
\maketitle

\vspace{-7mm}

\section{Introduction}

Let $C$ be a code and 
$C_\ell:=\{c\in C\mid \wt(c)=\ell\}$. 
In this paper, whenever $C_\ell$ is non-empty, we call it a shell of code $C$. 

Let $q$ be a prime power, and 
let 
$p$ be an odd prime where $q$ is a quadratic residue modulo $p$ and is 
coprime with $p$. 
Let $Q_{q,p}$ be the quadratic residue code of length $p$ over $\FF_q$, and 
let $\widetilde{Q}_{q,p+1}$ be the extended quadratic residue code of
length $p+1$ over $\FF_q$. 
By the transitivity argument, 
shells of $\widetilde{Q}_{q,p+1}$ are known to support 
$2$-designs if $p \equiv 1\pmod{4}$ and 
$3$-designs if $p \equiv -1\pmod{4}$ (see Example~\ref{ex:AutQR}).
In the recent years, Bonnecaze--Sol\'{e} \cite{BS} found a $3$-design 
in $(\widetilde{Q}_{2,42})_{\ell}$ 
and Ishikawa \cite{Ishikawa} also found $3$-designs 
in $(\widetilde{Q}_{3,14})_{\ell}$ and $(\widetilde{Q}_{4,18})_{\ell}$.
Since they are the cases where $p \equiv 1\pmod{4}$, they do not follow from the transitivity argument.
We mention that the results of them are only for some special $\ell$ 
while the transitivity argument give results for all $\ell$.

In this paper, we show infinite series of $3$-designs in $\widetilde{Q}_{r^2,p+1}$  
that do not follow from the transitivity argument like the above three examples.
Many of them also do not follow from the Assmus--Mattson Theorem (see Theorem~\ref{thm: AM}). 
The main result of this paper is as follows.
%

\begin{thm}\label{thm:newmain}
Let $p$ be an odd prime. 
Let $\widetilde{Q}_{{r^2},p+1}$ be the extended quadratic residue code of length $p+1$ over $\FF_{r^2}$. 
If 
$p \equiv1 \pmod{4}$ and $r$ is not a quadratic residue modulo $p$, 
then, for $\ell\in \NN$, ${(\widetilde{Q}_{{r^2},p+1})}_\ell$ is a $3$-design 
whenever $(\widetilde{Q}_{{r^2},p+1})_\ell$ 
is non-empty. 
\end{thm}

\begin{cor}\label{cor:new}
Regarding Theorem~\ref{thm:newmain}, 
if $r=2$ or $r=3$, the following statements hold:
\begin{enumerate}
\item [(1)]
If $p\equiv -3\pmod{8}$, 
then, for $\ell\in \NN$, ${(\widetilde{Q}_{4,p+1})}_\ell$ is a $3$-design 
whenever $(\widetilde{Q}_{4,p+1})_\ell$ is non-empty.
\item [(2)]
If 
$p\equiv 5\pmod{12}$, 
then, for $\ell\in \NN$, ${(\widetilde{Q}_{9,p+1})}_\ell$ is a $3$-design 
whenever $(\widetilde{Q}_{9,p+1})_\ell$ is non-empty.
\end{enumerate}
\end{cor}

\color{black}
This paper is organized as follows. 
In Section~\ref{sec:pre}, 
we give definitions and some basic properties of the 
codes and $t$-designs,  
Jacobi polynomials, and harmonic weight enumerators 
used in this paper. 
In Section~\ref{sec:prel}, 
we show 
preliminary steps of proving Theorem~\ref{thm:newmain}.
In Section~\ref{sec:sub}, 
we give a proof of 
Theorem~\ref{thm:newmain}. 
Finally, in Section~\ref{sec:rem}, 
we show some examples and 
infinite series of 
$3$-designs in $\widetilde{Q}_{r^2,p+1}$. 
\color{black}
All computations presented in this paper were done using 
{\sc Magma}~\cite{Magma} and {\sc Mathematica}~\cite{Mathematica}.

\section{Preliminaries}\label{sec:pre}

\subsection{Codes and $t$-designs}

A linear code $C$ of length $n$ over $\FF_q$ is a linear subspace of $\FF_{q}^{n}$. 
The Euclidean inner product $({x},{y})_E$ on $\FF_q^n$ is given 
by
\[
(x,y)_E=\sum_{i=1}^nx_iy_i,
\]
where $x,y\in \FF_q^n$ with $x=(x_1,x_2,\ldots, x_n)$ and 
$y=(y_1,y_2,\ldots, y_n)$. 
The Euclidean dual of a linear code $C$ is defined as follows: 
\[
C^{\perp,E}=\{{y}\in \FF_{q}^{n}\mid ({x},{y})_E =0 \text{ for all }{x}\in C\}. 
\]
A linear code $C$ is called Euclidean self-dual if $C=C^{\perp,E}$.

In the study of linear codes over $\FF_{r^2}$ 
for a prime $r$, 
we often use another inner product called the Hermitian inner product. 
The Hermitian inner product $(x,y)_H$ on $\FF_{r^2}^n$ is given 
by
\[
(x,y)_H 
= \sum_{i=1}^nx_i{y_i}^r,
\]
where 
$x,y\in \FF_{r^2}^n$ with $x=(x_1,x_2,\ldots, x_n)$ and $y=(y_1,y_2,\ldots, y_n)$. 
The Hermitian dual of a linear code $C$ over $\FF_{r^2}$ is defined as follows: 
\[
C^{\perp,H}=\{{y}\in \FF_{r^2}^n \mid ({x},{y})_H =0 \text{ for all }{x}\in C\}. 
\]
A linear code $C$ is called $C$ Hermitian self-dual if $C=C^{\perp,H}$. 
For details of the Hermitian inner product, see \cite{HP}. 
For $x \in\FF_q^n$,
the weight $\wt(x)$ is the number of its nonzero components. 
\color{black}

Let $C$ be a linear code of length $n$ over $\FF_q$.  
We consider creating longer codes by adding a new coordinate. 
The extended code $\widetilde{C}$ is defined as follows:
\[
\widetilde{C}
=\{(c_1,\ldots,c_n,c_{n+1})\in \FF_q^{n+1}\mid (c_1,\ldots,c_n)\in C,
\sum_{i=1}^{n+1} c_i=0\}. 
\]
We label the new coordinate $\infty$ and 
assume that the coordinates of 
$\widetilde{C}$ 
are labeled as $\{0,1,\ldots,n-1,\infty\}$.

Let $C$ be a cyclic code of length $n$ over $\FF_q$. 
That is, if $(c_1,c_2,\ldots,c_n)\in C$, then 
$(c_{n},c_1,\ldots,c_{n-1})\in C$. 
Then 
$C$ corresponds to an ideal $(g(x))$ of 
\[
\FF_q[x]/(x^n-1). 
\]
We call $g(x)$ a generator polynomial of $C$. 
Let $p$ be an odd prime where $q$ is a quadratic residue modulo $p$. 
The quadratic residue code $Q_{q,p}$ over $\FF_q$ is 
a cyclic code of length $p$, 
which is generated by 
\[
\prod_{\ell \in (\FF_p^\ast)^2}(x-\alpha^\ell), 
\]
where $(\FF_p^\ast)^2=\{\ell^2\mid \ell\in \FF_p^\ast\}$ and $\alpha$ is a primitive root of order $p$. 
We note that quadratic residue codes over $\FF_{r^2}$ exist for any odd prime length greater than $r$, 
as $r^2$ is obviously a quadratic residue modulo any odd prime $p$. 
For details of the quadratic residue codes, 
see \cite{{DC},{L}}. 

The quadratic residue code is a special case of duadic code.
Before showing the definition of duadic code, we recall the concepts of even-like and odd-like.
A vector $x=(x_1,x_2,\ldots, x_n) \in \FF_q^{n}$ is even-like if 
$\sum_{i=1}^n x_i =0$ and is odd-like otherwise. In addition, we call a code even-like if it has only even-like codewords and odd-like if it is not even-like.
For details of even-like and odd-like, see \cite{HP}. 
Duadic code is defined as follows:
\begin{df}[{{\cite[Definition 4.1]{DMS}}}]\label{ex:DDC}
Let $e_1(x)$ and $e_2(x)$ be a pair of even-like idempotents and let $C_1= \langle e_1(x) \rangle$ and $C_2= \langle e_2(x) \rangle$.
The codes $C_1$ and $C_2$ form a pair of even-like duadic codes of length $n$ over $\FF_q$ if the
following properties are satisfied:
\begin{enumerate}
\item [(a)]
the idempotents satisfy
$e_1(x)+e_2(x) = 1 - \bar{j}(x)$,
\item [(b)]
there is a multiplier $\mu_a$ such that 
$C_1 \mu_a = C_2$ and $C_2 \mu_a = C_1$,
\end{enumerate}
where $\bar{j}(x)=\frac{1}{n} (1+x+x^2+ \cdots +x^{n-1})$, and 
$\mu_a$ is a function on $\{0,1,2,\cdots,n-1\}$ by $i\mu_a \equiv ia\pmod{n}$.
The pair of even-like codes $C_1$ and $C_2$ is associated with a pair of odd-like duadic codes $D_1= \langle 1-e_2(x) \rangle$ and $D_2= \langle 1-e_1(x) \rangle$.
We say that the multiplier $\mu_a$ gives a splitting for
the even-like duadic codes or for the odd-like duadic codes.
\end{df}

\begin{lem}[{{\cite[Theorem 6.1.3 (ii),(ix)]{HP}}}]
\label{lem: cplus1}
Let $C_1=\langle e_1(x) \rangle$ and $C_2=\langle e_2(x) \rangle$ be a pair of even-like 
duadic codes of length $n$ over $\FF_q$. 
Let $D_1$ and $D_2$ be the associated odd-like duadic codes. 
Then, 
\begin{itemize}
  \item[(1)]
  $C_1 \cap C_2 = \{0\}. $
  \item[(2)] 
$D_i = C_i + \langle \bar{j}(x) \rangle 
=  \langle \bar{j}(x) + e_i(x) \rangle \,\,\text{for}\,\, i=1,2. $
  \end{itemize}
\end{lem}

\begin{rem}\label{rem: gmatrix}
Regarding Lemma~\ref{lem: cplus1}, 
$\bar{j}(x)$ corresponds to all-one vector $\mathbf{1}_n$ 
for all $q$ and $n$. 
Thus, Lemma~\ref{lem: cplus1} $(2)$
gives a generator matrix. 
Let $G_i$ be a generator matrix for $C_i$ for $i=1,2$,  and let $\widetilde{D}_i$ be the extended code of $D_i$ for $i=1,2$. 
Then, a generator matrix for $\widetilde{D}_i$ is as follows: 
\begin{equation*}
\begin{blockarray}{ccccc}
0 &1 &\cdots &n-1 &\infty     \\ \cline{1-5}
\begin{block}{(cccc|c)}
  &  &       &    &0          \\
  &  &{\huge{G_i}} & &\vdots   \\
  &  &       &    &0           \\ \cline{1-5}
1 &1 &\cdots &1   &k           \\
\end{block}
\end{blockarray}
\ \;,
\end{equation*}
where $\sum_{i=0}^{n-1} 1 + k =0.$
\end{rem}


\color{black}
Let $C$ be a code of length $n$. 
Then the symmetric group $S_n$ acts on the $n$ coordinates of $C$. 
The automorphism group $\PAut(C)$ of $C$ is a subgroup of $S_n$ such that 
\[
\PAut(C):=\{\sigma\in S_n\mid C^\sigma=C\}, 
\]
where 
\[
C^\sigma
:=
\{(c_{\sigma(1)},\ldots,c_{\sigma(n)})\mid (c_1,\ldots,c_n)\in C\}. 
\]

\begin{ex}\label{ex:AutQR}
Let $Q_{q,p}$ be the quadratic residue code of
length $p$ over $\FF_q$. 
We consider extending $Q_{q,p}$ when $Q_{q,p}$ is odd-like.
Let $\widetilde{Q}_{q,p+1}$ be the extended quadratic 
residue code of length $p+1$ over $\FF_q$. 
Then, 
the automorphism group of 
$\widetilde{Q}_{q,p+1}$ includes $PSL_2(p)$ 
\cite{HP}
(see also 
\cite{{assmus-mattson},{Blahut},{E},{Huffman}}). 
We identify the coordinates of 
$\widetilde{Q}_{q,p+1}$, $\{0,1,\ldots,p-1,\infty\}$ with $PG(1,p)$. 
Then, the action of $PGL_2(p)$ on $PG(1,p)$ is $3$-transitive 
(see \cite[Propositions 4.6 and 4.8]{BJL}) and
the action of $PSL_2(p)$ on $PG(1,p)$ 
is $2$-homogeneous (see \cite{E}). 
Thus, for $\ell\in \NN$, ${(\widetilde{Q}_{q,p+1})}_\ell$ is a $2$-design 
whenever $(\widetilde{Q}_{q,p+1})_\ell$ is non-empty. 
In addition, if $p\equiv -1\pmod{4}$, 
then the action of $PSL_2(p)$ on $PG(1,p)$ 
is $3$-homogeneous (see \cite{BR}). 
Thus, if $p\equiv -1\pmod{4}$, for $\ell\in \NN$, ${(\widetilde{Q}_{q,p+1})}_\ell$ is a $3$-design 
whenever $(\widetilde{Q}_{q,p+1})_\ell$ is non-empty.
However, if $p\equiv 1\pmod{4}$, 
\color{black}
the action of $PSL_2(p)$ on $PG(1,p)$ 
is not 3-homogeneous 
(see \cite{BR}). 
Indeed, 
let $a$ be a generator of $\FF_p^\ast$. 
Then 
\[
\binom{X}{3}
=PSL_2(p)\{0,1,\infty\}\sqcup PSL_2(p)\{0,a,\infty\}. 
\]
\end{ex}

\color{black}
A $t$-$(v,h,\lambda)$ design  (or $t$-design for short)
is a pair 
$\mathcal{D}=(\Omega,\mathcal{B})$, where $\Omega$ is a set of points of 
cardinality $v$, and $\mathcal{B}$ is a collection of $h$-element subsets
of $\Omega$ called blocks, with the property that any $t$ points are 
contained in precisely $\lambda$ blocks.

The support of a vector 
${x}:=(x_{1}, \dots, x_{n})$, 
$x_{i} \in \GF_{q}$, is 
the set of indices of its nonzero coordinates: 
${\rm supp} ({x}) = \{ i \mid x_{i} \neq 0 \}$\index{$supp (x)$}.
Let $\Omega:=\{1,\ldots,n\}$ and 
$\mathcal{B}(C_\ell):=\{\supp({x})\mid {x}\in C_\ell\}$. 
Then, for a code $C$ of length $n$, 
we say that $C_\ell$ is a $t$-design if 
$(\Omega,\mathcal{B}(C_\ell))$ is a $t$-design. 
We next show a theorem that gives a sufficient condition for a code to hold designs, 
which is called the Assmus--Mattson Theorem. 




\begin{thm}[{{\cite[Theorem 4.2]{assmus-mattson}}, {\cite[Theorem 2.2]{Tanabe}}}]
\label{thm: AM}
Let $C$ be a linear code of length $n$ over $\FF_q$ with minimum weight $d$, and let $e$ be the minimum weight of $C^{\perp,E}$.
Let $t$ be an integer less than $d$, and let $B_i$ be the number of vectors of weight $i$ in $C^{\perp,E}$.
Let $w_0$ and $w_1$ be the largest integers satisfying
\begin{align*}
w_0-\left[\frac{w_0+(q-2)}{q-1}\right]<d,\,\, 
w_1-\left[\frac{w_1+(q-2)}{q-1}\right]<e,
\end{align*}
where, if $q=2$, we have $w_0=w_1:=n$. Suppose that $s=|\{i \mid B_i \ne 0,\, 0<i \le n-t\}|$ is at most $d-t$.
Then, for each weight $i$ with $d \le i \le w_0$ 
(resp. $e \le i \le min\{n-t,w_1\}$), the vectors of weight $i$ in $C$ (resp. $C^{\perp,E}$) support a $t$-design. 
\end{thm}


\subsection{Jacobi polynomials}

Let $C$ be a code of length $n$ over $\FF_q$ and $T\subset [n]:=\{1,\ldots,n\}$. 
Then the Jacobi polynomial of $C$ with $T$ is defined as follows \cite{Ozeki}:
\[
J_{C,T} (w,z,x,y) :=\sum_{c\in C}w^{m_0(c)} z^{m_1(c)}x^{n_0(c)}y^{n_1(c)}, 
\]
where for $c=(c_1,\ldots,c_n)$, 
\begin{align*}
m_0(c)&=|\{j\in T\mid c_j=0 \}|,\\
m_1(c)&=|\{j\in T\mid c_j \neq 0 \}|,\\
n_0(c)&=|\{j\in [n]\setminus T\mid c_j=0 \}|,\\
n_1(c)&=|\{j\in [n]\setminus T\mid c_j \neq 0 \}|.\\
\end{align*}
\begin{rem}\label{rem:hom}

It is easy to see that 
$C_\ell$ is a $t$-design 
if and only if 
the coefficient of $z^{t}x^{n-\ell}y^{\ell-t}$ 
in $J_{C,T}$ is independent of the choice of $T$ with $|T|=t$. 
Moreover, for all $\ell\in \NN$, $C_\ell\cup C_\ell^{\perp,\ast}$ \,$(\ast:E,H)$ is a $t$-design 
if and only if 
$J_{C,T}+J_{C^{\perp,\ast},T}$ \,$(\ast:E,H)$ 
is independent of the choice of $T$ with $|T|={t}$. 

\end{rem}

\subsection{Harmonic weight enumerators}\label{sec:Har}


In this subsection, we review the concept of 
harmonic weight enumerators.

Let $\Omega=\{1, 2,\ldots,n\}$ be a finite set (which will be the set of coordinates of the code) and 
let $X$ be the set of its subsets,  while for all $k= 0,1,\dots, n$, 
$X_{k}$ is the set of its $k$-subsets.
We denote by $\R X$ and $\R X_k$ the 
real vector spaces spanned by the elements of $X$
and $X_{k}$, respectively.
An element of $\R X_k$ is denoted by
$$f=\sum_{z\in X_k}f(z)z$$
and is identified with the real-valued function on $X_{k}$ given by 
$z \mapsto f(z)$. 

An element $f\in \R X_k$ can be extended to an element $\widetilde{f}\in \R X$ by setting
$$\widetilde{f}(u)=\sum_{z\in X_k, z\subset u}f(z)$$
for all $u \in X$. If an element $g \in \R X$ is equal to $\widetilde{f}$ 
for some $f \in \R X_{k}$, then we say that $g$ has degree $k$. 
The differentiation $\gamma$ is the operator on $\RR X$ defined by linearity from 
$$\gamma(z) =\sum_{y\in X_{k-1},y\subset z}y$$
for all $z\in X_k$ and for all $k=0,1, \ldots, n$, and $\Harm_{k}$ is the kernel of $\gamma$:
$$\Harm_k =\ker(\gamma|_{\R X_k}).$$

The symmetric group $S_n$ acts on $\Omega$ and 
the automorphism group $\PAut(\mathcal{B})$ of $\mathcal{B}$ is 
the subgroup of $S_n$ such that 
\[
\PAut(\mathcal{B}):=\{\sigma\in S_n\mid \mathcal{B}^\sigma=\mathcal{B}\}, 
\]
where 
\[
\mathcal{B}^\sigma:=\{ \{\sigma(b_1),\ldots,\sigma(b_m)\}\mid 
\{b_1,\ldots,b_m\}\in \mathcal{B}\}. 
\]
Let $G$ be a subgroup of $\PAut(\mathcal{B})$. 
Then, $G$ acts on $\Harm_k$ through the above action and 
we denote by $\Harm_k^{G}$ the set of the invariants of $G$: 
\[
\Harm_k^{G}=\{f\in {\rm Harm}_k\mid f^\sigma=f, \forall \sigma\in G\}, 
\]
where $f^\sigma$ is defined by linearity from 
\[
\{{i_1},\ldots, {i_k}\}^\sigma=\{{\sigma(i_1)},\ldots, {\sigma(i_k)}\}. 
\]
\begin{thm}[{{\cite[Theorem 2.5]{AMMN}}}]
\label{thm:design-aut}
A set $\mathcal{B} \subset X_{m}$ $($where $m \leq n$$)$ of blocks is a $t$-design 
if and only if $\sum_{b\in \mathcal{B}}\widetilde{f}(b)=0$ 
for all $f\in \Harm_k^{G}$, $1\leq k\leq t$. 
\end{thm}




In \cite{Bachoc-2}, 
the harmonic weight enumerator associated with a linear code $C$ over $\FF_q$ was defined as follows. 
\begin{df}
Let $C$ be a linear code of length $n$ over $\FF_q$ and let $f\in\Harm_{k}$. 
The harmonic weight enumerator associated with $C$ and $f$ is

$$w_{C,f}(x,y)=\sum_{{c}\in C}\widetilde{f}({\supp (c)})x^{n-\wt({c})}y^{\wt({c})}.$$
\end{df}


\begin{rem}\label{rem:hom-2}
It follows from Theorem 
\ref{thm:design-aut} that 
$C_\ell$ is a $t$-design if and only if 
the coefficient of $x^{n-\ell}y^\ell$ in $w_{C,f}(x,y)$ vanishes 
for all $f\in \Harm_k^{\PAut(C)}\ (1\leq k\leq t)$. 
Moreover, for all $\ell\in \NN$, $C_\ell\cup C_\ell^{\perp,\ast}$ \,$(\ast:E,H)$ is a $t$-design 
if and only if 
$w_{C,f}+w_{C^{\perp,\ast},f}=0$ \,$(\ast:E,H)$
for all $f\in \Harm_k^{\PAut(C)}\ (1\leq k\leq t)$. 
\end{rem}

\section{Preliminary steps of proving Theorem~\ref{thm:newmain}}\label{sec:prel}
In this section, we show Lemma~\ref{lem:main}, Corollary~\ref{cor:sub}, and Proposition~\ref{prop:intersect} 
and give proofs of them. 
These are used for proving Theorem~\ref{thm:newmain}. 
First, we show Lemma~\ref{lem:main} and Corollary~\ref{cor:sub}. 
They are similar in content to \cite{AMMN}, but we give proofs of them for readers' convenience. 
We use $\sqcup$ to denote disjoint union and 
a code $C$ is isodual if $C$ and $C^{\perp, \ast}\,(\ast:E, H)$ are equivalent.



\begin{lem}\label{lem:main}
Let $C$ be an isodual code of length $n$ over $\FF_q$, 
$G$ be a subgroup of $\PAut(C)$, and 
$X:=\{1,\ldots,n\}$.
\color{black}
Let $\sigma \in S_n$ such that 
$C^{\perp, \ast}=C^\sigma \,(\ast:E, H)$. 
Then, $G$ acts on $\binom{X}{t}$ and 
we assume that 
$G$ has two orbits, namely,  
\[
\binom{X}{t}=GT_1\sqcup GT_2,
\]
such that $(GT_1)^\sigma=GT_2$. 
Then the following statements hold: 
\begin{enumerate}
\item [(1)]
$J_{C,T}+J_{C^{\perp, \ast},T}\, (\ast:E, H)$  
is independent of the choice of $T$ with $|T|={t}$. 

\item [(2)]
Let $f$ be a harmonic function of degree $t$, 
which is an invariant of $\PAut(C)$. 
Then we have 
\[
w_{C,f}+w_{C^{\perp, \ast},f}=0\, (\ast:E, H). 
\]

\end{enumerate}
\end{lem}

Applying Lemma~\ref{lem:main}, 
we have the following corollary: 

\begin{cor}\label{cor:sub}
Let $p$ be an odd prime where $q$ is a quadratic residue modulo $p$, and
let $\widetilde{Q}_{q,p+1}$ be the extended quadratic residue code of
length $p+1$ over $\FF_q$. 
Then 
for $\ell\in \NN$, $(\widetilde{Q}_{q,p+1})_\ell \cup {(\widetilde{Q}_{q,p+1}^{\perp, \ast})}_\ell\, (\ast:E, H)$ 
is a $3$-design 
whenever 
it is non-empty.
\end{cor}

\begin{proof}[Proof of Lemma~\ref{lem:main}]

\begin{enumerate}
\item [(1)]

Recall that 
$C^{\sigma}=C^{\perp,\ast}$ \,$(\ast:E,H)$
and $(GT_1)^\sigma=GT_2$. 
We note that 
for all $T\in GT_i\ (i\in\{1,2\})$, $J_{C,T}=J_{C,T_i}$. 
Then, for any $T\in \binom{X}{t}$, 
\begin{align*}
J_{C,T}+J_{C^{\perp,\ast},T}&=J_{C,T}+J_{C^\sigma,T}\\
&=J_{C,T}+J_{C,T^{\sigma^{-1}}}\\
&=J_{C,T_1}+J_{C,T_2}\,(\ast:E,H). 
\end{align*}
Hence, $J_{C,T}+J_{C^{\perp,\ast},T}$ \,$(\ast:E,H)$
is independent 
of the choice of $T$ with $|T|={t}$. 


\item [(2)]
For $f\in {\rm Harm}_{t}^{G}$, 
$f$ is written as a linear combination of 
$R({T_1})$ and $R({T_2})$, 
where $R$ is the Reynolds operator: for $T\in \binom{X}{t}$, 
\[
R({T})=\frac{1}{|G|}\sum_{\sigma\in G}{T}^{\sigma}. 
\]
Based on the above assumption, $|GT_1|=|GT_2|$. 
Then, there exists a constant $c\in \RR$ such that 
\[
{\rm Harm}_{t}^{G}=\langle c(R({T_1})-R({T_2}))\rangle. 
\]
Let $f:=R({T_1})-R({T_2})$. 
Then 
\begin{align*}
w_{C,f}+w_{C^{\perp,\ast},f}&=w_{C,f}+w_{C^\sigma,f}\\
&=w_{C,f}+w_{C,f^{\sigma^{-1}}}\\
&=w_{C,f}+w_{C,-f}=0\,(\ast:E,H). 
\end{align*}
\end{enumerate}
This completes the proof of Lemma~\ref{lem:main}. 
\end{proof}

\begin{proof}[Proof of Corollary~\ref{cor:sub}]

By {{\cite[Corollary 2.10.1]{AK}}}, the extended quadratic residue code 
is an isodual code, so we apply Lemma~\ref{lem:main} to it.
Let $p$ be an odd prime where $q$ is a quadratic residue modulo $p$,
and let $\widetilde{Q}_{q,p+1}$ be the extended quadratic residue code of
length $p+1$ over $\FF_q$. 
Then, if $p\equiv 1\pmod{4}$, by Example~\ref{ex:AutQR}, 
\[
\PAut(\widetilde{Q}_{q,p+1})\supset PSL_2(p) =:G. 
\]
Let $X:=\{0,1,\ldots,p-1,\infty\}$. 
Then, 
$G$ 
acts on $\binom{X}{3}$ and, 
by Example~\ref{ex:AutQR}
and \cite[P.235, 8A.8.]{I}, 
we have two orbits, namely, 
\[
\binom{X}{3}=GT_1\sqcup GT_2, 
\]
that satisfy the assumption of Lemma~\ref{lem:main}. 
By Lemma~\ref{lem:main}, 
\[
J_{\widetilde{Q}_{q,p+1},T}+J_{\widetilde{Q}_{q,p+1}^{\perp,\ast},T} \,(\ast:E,H)
\] 
 is independent 
of the choice of $T$ with $|T|={3}$, 
and, for $f\in {\rm Harm}_3^{G}$, 
we have 
\[
w_{\widetilde{Q}_{q,p+1},f}+w_{\widetilde{Q}_{q,p+1}^{\perp,\ast},f}=0\,(\ast:E,H).  
\]
Since $(\widetilde{Q}_{q,p+1})_\ell$ and $(\widetilde{Q}_{q,p+1}^{\perp,\ast})_\ell$
\,$(\ast:E,H)$ are 2-designs, 
for $f\in {\rm Harm}_k^{G}\ (1\leq k\leq 2)$, 
we have 
\[
w_{\widetilde{Q}_{q,p+1},f}+w_{\widetilde{Q}_{q,p+1}^{\perp,\ast},f}=0\,(\ast:E,H).  
\]
Then, by Remark~\ref{rem:hom} or \ref{rem:hom-2}, 
the proof is complete if $p\equiv 1\pmod{4}$. 
On the other hand, if $p\equiv -1\pmod{4}$, 
for all $\ell \in \NN$, $(\widetilde{Q}_{q,p+1})_\ell$ is a $3$-design whenever $(\widetilde{Q}_{q,p+1})_\ell$ is non-empty
by Example~\ref{ex:AutQR}.
Therefore, the assertion of this corollary is evident.
This completes the proof of Corollary~\ref{cor:sub}.

\end{proof}


Next, we show Proposition~\ref{prop:intersect}. 
Then, 
the set of shells of code is not allowed to have repeated blocks when we have the inequalities in Theorem~\ref{thm: AM}, 
whereas we assume that $t$-designs in $\widetilde{Q}_{q,p+1}$ allow the existence of repeated blocks.

\begin{prop}\label{prop:intersect}
Let $q$ be the square of a prime number, namly, let $q=r^2$. 
Let $\widetilde{Q}_{r^2,p+1}$ 
be the extended quadratic residue code of length $p+1$ over $\FF_{r^2}$.
Then we have 

\begin{equation*}
\widetilde{Q}_{r^2,p+1} \cap \,\widetilde{Q}_{r^2,p+1}^{\perp,H}
=
\begin{cases}
\{0\} & \text{if $(i)\ \;:$ $-r \in Q 
\,\wedge 
\,p\not\equiv -1\!\!\!\pmod{r}$, }\\
\{0,\mathbf{1}_{p+1}\} & \text{if $(ii)\,\,:$ $-r \in Q 
\,\wedge 
\,p\equiv -1\!\!\!\pmod{r}$,} \\
\widetilde{Q_{r^2,p}^{\perp,H}} 
& \text{if $(iii):$ $-r \in N 
\,\wedge 
\,p\not\equiv -1\!\!\!\pmod{r}$,} \\
\widetilde{Q}_{r^2,p+1} 
& \text{if $(iv)\::$ $-r \in N 
\,\wedge 
\,p\equiv -1\!\!\!\pmod{r}$,} 
\end{cases}
\end{equation*}
where $Q$ is the set of quadratic residues modulo $p$, $N$ is the set of non-quadratic residues modulo $p$, 
$\mathbf{1}_{p+1}$ is the all-one vector of length $p+1$, 
and 
$\widetilde{Q_{r^2,p}^{\perp,H}}$ 
is the extended code of ${Q_{r^2,p}^{\perp,H}}$. 
\end{prop}

We now show a proof of Proposition~\ref{prop:intersect}. 
Let $C$ be a code and $M$ be a matrix,  
and let $CM=\{c \cdot M \mid c \in C\}$. 
We begin by proving the following lemma: 


\begin{lem}
\label{lem: duadicdual}
  Let $D_1$ and $D_2$ be a pair of odd-like duadic codes of length $p$ over $\FF_{r^2}$, and let $\widetilde{D}_i$ be the extended code of $D_i$ for $i=1,2$. Then the following statement hold:
\begin{itemize}
 \item[(1)]
  If $D_i\mu_{-r}=D_i$ for $i=1,2$, then we have 
  $\widetilde{D}_1^{\perp,H} = \widetilde{D}_2M$ and
  $\widetilde{D}_2^{\perp,H} = \widetilde{D}_1M$ with $M=diag(1,1,\ldots, 1,-1/p)$.
  \item[(2)] 
  If $\mu_{-r}$ gives the splitting for $D_1$ and $D_2$, then 
  we have $\widetilde{D}_i^{\perp,H} = \widetilde{D}_iM$ with $M=diag(1,1,\ldots, 1,-1/p)$ for $i=1,2$.
\end{itemize}
\end{lem}

\begin{proof}[Proof of Lemma~\ref{lem: duadicdual}]
\begin{itemize}
\item[(1)]
By Remark~\ref{rem: gmatrix}, generator matrices for $\widetilde{D}_1$ and $\widetilde{D}_2M$ are as follows:
\begin{equation}
\begin{blockarray}{ccccc}
0 &1 &\cdots &p-1 &\infty     \\ \cline{1-5}
\begin{block}{(cccc|c)}
  &  &       &    &0          \\
  &  &{\huge{G_1}} & &\vdots   \\
  &  &       &    &0           \\ \cline{1-5}
1 &1 &\cdots &1   &-p           \\
\end{block}
\end{blockarray}
\ \;,\ \;
\begin{blockarray}{ccccc}
0 &1 &\cdots &p-1 &\infty     \\ \cline{1-5}
\begin{block}{(cccc|c)}
  &  &       &    &0          \\
  &  &{\huge{G_2}} & &\vdots   \\
  &  &       &    &0           \\ \cline{1-5}
1 &1 &\cdots &1   &1           \\
\end{block}
\end{blockarray}
\ \;,
\end{equation}
where $G_i$ is a generator matrix for $C_i$ for $i=1,2$. 
By the proof of {{\cite[Proposition 4.8]{DMS}}}, 
if $D_1\mu_{-r}=D_1$, then $\widetilde{C}_1$ and $\widetilde{C}_2$ are Hermitian orthogonal to each other. 
In addition, $\hat{\mathbf{1}}=(1,1,\ldots,1,-p)$ is Hermitian orthogonal to 
both $\mathbf{1}_{p+1}=(1,1,\ldots,1,1)$ and $\widetilde{C}_2$.
Since 
$\dim(\widetilde{D}_1)=\dim(\widetilde{D}_2M)$, 
we have $\widetilde{D}_1^{\perp,H} = \widetilde{D}_2M$.
We also have $\widetilde{D}_2^{\perp,H} = \widetilde{D}_1M$ by the same proof.
  
  \item [(2)]
  By Remark~\ref{rem: gmatrix}, generator matrices for $\widetilde{D}_i$ and $\widetilde{D}_iM$ are as follows:
\begin{equation}
\begin{blockarray}{ccccc}
0 &1 &\cdots &p-1 &\infty     \\ \cline{1-5}
\begin{block}{(cccc|c)}
  &  &       &    &0          \\
  &  &{\huge{G_i}} & &\vdots   \\
  &  &       &    &0           \\ \cline{1-5}
1 &1 &\cdots &1   &-p           \\
\end{block}
\end{blockarray}
\ \;,\ \;
\begin{blockarray}{ccccc}
0 &1 &\cdots &p-1 &\infty     \\ \cline{1-5}
\begin{block}{(cccc|c)}
  &  &       &    &0          \\
  &  &{\huge{G_i}} & &\vdots   \\
  &  &       &    &0           \\ \cline{1-5}
1 &1 &\cdots &1   &1           \\
\end{block}
\end{blockarray}
\ \;,
\end{equation}
where $G_i$ is a generator matrix for $C_i$ for $i=1,2$. 
By the proof of {{\cite[Proposition 4.8]{DMS}}}, 
if $\mu_{-r}$ gives the splitting for $D_1$ and $D_2$, then 
$\widetilde{C}_i$ is Hermitian self-orthogonal.
In addition, $\hat{\mathbf{1}}=(1,1,\ldots,1,-p)$ is Hermitian orthogonal to 
both $\mathbf{1}_{p+1}=(1,1,\ldots,1,1)$ and $\widetilde{C}_i$.
Since 
$\dim(\widetilde{D}_i)=\dim(\widetilde{D}_iM)$, 
we have $\widetilde{D}_i^{\perp,H} = \widetilde{D}_iM$ for $i=1,2$.
\end{itemize}
\end{proof}

\begin{proof}[Proof of Proposition~\ref{prop:intersect}]
Let $D_1$ and $D_2$ be a pair of odd-like quadratic residue codes of length $p$ over $\FF_{r^2}$, and let $\widetilde{D}_i$ be the extended code of $D_i$ for $i=1,2$. 
We show a proof by dividing into cases where $-r \in Q$ and $-r \in N$.
\begin{itemize}
  \item [(a)]
If $-r \in Q$, then $D_i\mu_{-r}=D_i$ for $i=1,2$.
By Lemma~\ref{lem: duadicdual}$(1)$, 
we have 
  $\widetilde{D}_1^{\perp,H} = \widetilde{D}_2M$ and
  $\widetilde{D}_2^{\perp,H} = \widetilde{D}_1M$ with $M=diag(1,1,\ldots, 1,-1/p)$.
Then, if $-1/p=1$, namely, if $p \equiv -1\pmod{r}$, 
$M$ is equal to the identity matrix. 
Furthermore, we have $\widetilde{C}_1 \cap \widetilde{C}_2=\{0\}$ by Lemma~\ref{lem: cplus1} $(1)$. 
Thus, by $(3.1)$ in Lemma~\ref{lem: duadicdual}, we have
\begin{equation*}
\begin{cases}
\widetilde{D}_1 \cap \widetilde{D}_1^{\perp, H} = 
\widetilde{D}_1 \cap \widetilde{D}_2M = \{0\} &
\text{if $p \not\equiv -1\pmod{r}$,}\\
\widetilde{D}_1 \cap \widetilde{D}_1^{\perp, H} = 
\widetilde{D}_1 \cap \widetilde{D}_2 = \{0,\mathbf{1}_{p+1}\} &
\text{if $p \equiv -1\pmod{r}$.}
\end{cases}
\end{equation*}

Therefore, we have
\begin{equation*}
\begin{cases}
\widetilde{Q}_{r^2,p+1} \cap \widetilde{Q}_{r^2,p+1}^{\perp,H} = 
\{0\} &
\text{if $(i)$,}\\
\widetilde{Q}_{r^2,p+1} \cap \widetilde{Q}_{r^2,p+1}^{\perp,H} = \{0,\mathbf{1}_{p+1}\} &
\text{if $(ii)$.}
\end{cases}
\end{equation*}

\item [(b)]
If $-r \in N$, then $\mu_{-r}$ gives the splitting for $D_1$ and $D_2$. 
By Lemma~\ref{lem: duadicdual}$(2)$, 
we have $\widetilde{D}_i^{\perp,H} = \widetilde{D}_iM$ with $M=diag(1,1,\ldots, 1,-1/p)$ for $i=1,2$.
Then, if $-1/p=1$, namely, if $p \equiv -1\pmod{r}$, 
$M$ is equal to the identity matrix.
Thus, by $(3.2)$ in Lemma~\ref{lem: duadicdual}, we have
\begin{equation*}
\begin{cases}
\widetilde{D}_i \cap \widetilde{D}_i^{\perp, H} = 
\widetilde{D}_i \cap \widetilde{D}_iM = \widetilde{C}_i
&
\text{if $p \not\equiv -1\pmod{r}$,}\\
\widetilde{D}_i \cap \widetilde{D}_i^{\perp, H} = 
\widetilde{D}_i \cap \widetilde{D}_i = 
\widetilde{D}_i &
\text{if $p \equiv -1\pmod{r}$.}
\end{cases}
\end{equation*}

In addition, by {{\cite[Proposition 4.6]{DMS}}}, 
we have $C_i = D_i^{\perp, H}$ 
and $C_i$ is the even-like code, so 
we have $\widetilde{C}_i = \widetilde{D_i^{\perp,H}}$.
Therefore, we have
\begin{equation*}
\begin{cases}
\widetilde{Q}_{r^2,p+1} \cap \widetilde{Q}_{r^2,p+1}^{\perp,H} = 
\widetilde{Q_{r^2,p}^{\perp,H}} &
\text{if $(iii)$,}\\
\widetilde{Q}_{r^2,p+1} \cap \widetilde{Q}_{r^2,p+1}^{\perp,H} = \widetilde{Q}_{r^2,p+1} &
\text{if $(iv)$.}
\end{cases}
\end{equation*}
$(iv)$ is the case where $\widetilde{Q}_{r^2,p+1}$ is Hermitian self-dual.
\end{itemize}
This completes the proof of Proposition~\ref{prop:intersect}.
\end{proof}

\begin{rem}\label{rem:intersect}
With respect to Proposition~\ref{prop:intersect}, 
if $q$ is a prime power, namely, if $q=r^t$,
then we have 
\begin{equation*}
\widetilde{Q}_{q,p+1} \cap \,\widetilde{Q}_{q,p+1}^{\perp,E}
=
\begin{cases}
\{0\} & \text{if $(v)\ \;\,:$ $-1 \in Q 
\,\wedge 
\,p\not\equiv -1\!\!\!\pmod{r}$, }\\
\{0,\mathbf{1}_{p+1}\} & \text{if $(vi)\,\,\,:$ $-1 \in Q 
\,\wedge 
\,p\equiv -1\!\!\!\pmod{r}$,} \\
\widetilde{Q_{q,p}^{\perp,E}} 
& \text{if $(vii)\;:$ $-1 \in N 
\,\wedge 
\,p\not\equiv -1\!\!\!\pmod{r}$,} \\
\widetilde{Q}_{q,p+1} 
& \text{if $(viii):$ $-1 \in N 
\,\wedge 
\,p\equiv -1\!\!\!\pmod{r}$.} 
\end{cases}
\end{equation*}
It can be proven like Proposition~\ref{prop:intersect}, 
using {{\cite[Theorem 6.4.2]{HP}}, {\cite[Theorem 6.4.3]{HP}}, and {\cite[Theorem 6.4.12]{HP}}}.  
\end{rem}

\section{Proof of Theorem~\ref{thm:newmain}}\label{sec:sub}

In this section, 
we give a proof of Theorem~\ref{thm:newmain}, using Corollary~\ref{cor:sub}, Proposition~\ref{prop:intersect}, and Lemma~\ref{lem: duadicdual}. 
Before proving Theorem~\ref{thm:newmain}, we state a lemma: 

\begin{lem}[{{\cite[Lemma 6.2.4]{HP}}}]\label{lem: -1QN}
Let $p$ be an odd prime. Then, $-1$ is a quadratic residue modulo $p$
if and only if $p \equiv 1\pmod{4}$.
\end{lem}

%

\begin{proof}[Proof of Theorem~\ref{thm:newmain}]
Let $p$ be an odd prime and $p \equiv1 \pmod{4}$, 
and let $r$ be not a quadratic residue modulo $p$. 
Then, $-1$ 
is a quadratic residue modulo $p$ 
by Lemma~\ref{lem: -1QN}. 
Therefore, 
$-r$ is not a quadratic residue modulo $p$, namely, 
$-r \in N$. 
Let $\widetilde{Q}_{{r^2},p+1}$ be the extended quadratic residue code of length $p+1$ over $\FF_{r^2}$.
If $p \equiv-1 \pmod{r}$, 
$\widetilde{Q}_{r^2,p+1}$ is Hermitian self-dual 
by Proposition~\ref{prop:intersect}. Therefore, 
by Corollary~\ref{cor:sub}, 
for $\ell\in \NN$, ${(\widetilde{Q}_{{r^2},p+1})}_\ell$ is obviously a $3$-design 
whenever $(\widetilde{Q}_{{r^2},p+1})_\ell$ is non-empty.
On the other hand, if $p \not\equiv-1 \pmod{r}$, 
$\widetilde{Q}_{r^2,p+1}$ is not Hermitian self-dual.
Let $D_1$ and $D_2$ be a pair of odd-like quadratic residue codes of length $p$ over $\FF_{r^2}$.
For $i=1,2$, let $C_i$ be an even-like subcode of $D_i$.
By the proof of Lemma~\ref{lem: duadicdual} $(2)$,  
\begin{align*}
\widetilde{D}_i &= 
\{c+k\hat{\mathbf{1}} \mid 
c \in \widetilde{C}_i, k=0,1,\ldots, r-1\}, \\ 
\widetilde{D}_i^{\perp,H} &= 
\{c+k\mathbf{1}_{p+1} \mid 
c \in \widetilde{C}_i, k=0,1,\ldots, r-1\}. 
\end{align*}
Since $C_i$ is even-like for $i=1,2$, the extended coordinate of codewords in $\widetilde{C}_i$ is always equal to zero. 
Thus, for all $\ell$,
\begin{align*}
\mathcal{B}((\widetilde{Q}_{{r^2},p+1})_\ell):&=\{\supp({x})\mid {x}\in (\widetilde{Q}_{{r^2},p+1})_\ell\} \\
&=\{\supp({x})\mid {x}\in (\widetilde{Q}_{{r^2},p+1}^{\perp,H})_\ell\}
=:\mathcal{B}((\widetilde{Q}_{{r^2},p+1}^{\perp,H})_\ell).
\end{align*}
Therefore, 
by Corollary~\ref{cor:sub}, 
for $\ell\in \NN$, ${(\widetilde{Q}_{{r^2},p+1})}_\ell$ is a $3$-design 
whenever $(\widetilde{Q}_{{r^2},p+1})_\ell$ is non-empty.
This completes the proof of Theorem~\ref{thm:newmain}.
\end{proof}

We are ready to prove Corollary~\ref{cor:new}.
\begin{proof}[Proof of Corollary~\ref{cor:new}]
\begin{enumerate}
\item[(1)]
Let 
$p\equiv -3\pmod{8}$.
Then, 
$p\equiv 1\pmod{4}$ and 
$2$ is not a quadratic residue modulo $p$ by {{\cite[Lemma 6.2.5]{HP}}}. 
Therefore, by Theorem~\ref{thm:newmain}, for $\ell\in \NN$, ${(\widetilde{Q}_{4,p+1})}_\ell$ is a $3$-design 
whenever $(\widetilde{Q}_{4,p+1})_\ell$ is non-empty. 

%
\item[(2)]
Let 
$p\equiv 5\pmod{12}$. 
Then, 
$p\equiv 1\pmod{4}$ and 
$3$ is not a quadratic residue modulo $p$ by {{\cite[Lemma 6.2.9]{HP}}}. 
Therefore, by Theorem~\ref{thm:newmain}, for $\ell\in \NN$, ${(\widetilde{Q}_{9,p+1})}_\ell$ is a $3$-design 
whenever $(\widetilde{Q}_{9,p+1})_\ell$ is non-empty.
\end{enumerate}
This completes the proof of Corollary~\ref{cor:new}.
\end{proof}



The proof of Theorem~\ref{thm:newmain} showed that 
for $\ell\in \NN$, 
$(\widetilde{Q}_{{r^2},p+1})_\ell$ is a $3$-design 
in the cases of $(iii)$ and $(iv)$ in Proposition~\ref{prop:intersect}. 
By the same argument,
for $\ell\in \NN$, 
$(\widetilde{Q}_{{q},p+1})_\ell$ is a $3$-design 
in the cases of $(vii)$ and $(viii)$ in Remark~\ref{rem:intersect}. 
In other words, for $\ell\in \NN$, 
$(\widetilde{Q}_{{q},p+1})_\ell$ is a $3$-design 
if $-1 \in N$, that is, if $p \equiv -1\pmod{4}$. 
This result corresponds to the result followed by the transivity argument. 

\section{$3$-designs in the extended quadratic residue codes over $\FF_{r^2}$}\label{sec:rem}

Let $p$ be an odd prime and  
let $\widetilde{Q}_{{r^2},p+1}$ be the extended quadratic residue code of length $p+1$ over $\FF_{r^2}$. 

\subsection{Examples}
In this subsection, we demonstrate that $3$-designs obtained from Theorem~\ref{thm:newmain} do not follow from the transitivity argument with some examples. 
We also demonstrate that some examples do not follow from the Assmus--Mattson Theorem.


\begin{itemize}

  \item[(1)]
  Let $r=2$ and $p=37$. 
  The weight distribution of $\widetilde{Q}_{4,38}^{\perp,E}$ is 
  \[
  \{0\} \cup \{2i \mid 6 \le i \le 19,\, i\in \NN\}.
  \]
For $t=3$, there are twelve weights below $35$, and $d-t=9$. 
  Thus, the Assmus--Mattson Theorem is inapplicable in this case. 
However, since $37 \equiv -3 \pmod{8}$, we can show that $(\widetilde{Q}_{4,38})_\ell$ is a $3$-design by Corollary~\ref{cor:new} $(1)$. 

  \item[(2)]
  Let $r=2$ and $p=29$. 
  The weight distribution of $\widetilde{Q}_{4,30}^{\perp,E}$ is 
  \[
  \{0\} \cup \{2i \mid 6 \le i \le 15,\, i\in \NN\}.
  \]
For $t=3$, there are eight weights below $27$, and $d-t=9$. 
  Thus, the Assmus--Mattson Theorem is applicable in this case 
  and we can show that for $\ell=12,14,16$, $(\widetilde{Q}_{4,30})_\ell$ is a $3$-design by the Assmus--Mattson Theorem. 
  Furthermore, since $29 \equiv -3 \pmod{8}$, we can show that for $\ell \in \NN$, $(\widetilde{Q}_{4,30})_\ell$ is a $3$-design by Corollary~\ref{cor:new} $(1)$. 

  \item[(3)]
  Let $r=5$ and $p=17$. 
  The weight distribution of $\widetilde{Q}_{25,18}^{\perp,E}$ is 
  \[
  \{0\} \cup \{i \mid 10 \le i \le 18,\, i\in \NN\}.
  \]
For $t=3$, there are six weights below $15$, and $d-t=7$. 
  Thus, the Assmus--Mattson Theorem is applicable in this case 
  and we can show that for $\ell=10$, $(\widetilde{Q}_{25,18})_\ell$ is a $3$-design by the Assmus--Mattson Theorem. 
  Furthermore, since $17 \equiv 1\pmod{4}$ and $5$ is not a quadratic residue modulo $17$, we can show that for $\ell \in \NN$, $(\widetilde{Q}_{25,18})_\ell$ is a $3$-design by Theorem~\ref{thm:newmain}. 
\end{itemize}

Since $p\equiv 1\pmod{4}$ in all the above examples, 
it is evident that they do not follow from the transitivity argument by Example~\ref{ex:AutQR}. 
In addition, it becomes also evident that as $r$ increases, Theorem~\ref{thm:newmain} becomes effective.

\subsection{Infinite series of $3$-designs in the extended quadratic residue code over $\FF_{r^2}$}

\color{black}
Finally, we show infinite series of $3$-designs in $\widetilde{Q}_{r^2,p+1}$. 
By Theorem~\ref{thm:newmain}, 
for $\ell\in \NN$, ${(\widetilde{Q}_{r^2,p+1})}_\ell$ is a $3$-design 
whenever $(\widetilde{Q}_{r^2,p+1})_\ell$ is non-empty 
if $r$ and $p$ satisfy the following conditions:
\begin{equation}
\begin{cases}
(i)\,\,p \equiv 1\pmod{4},\\
(ii)\,\, \text{$r$ is not a quadratic residue modulo $p$}.\\
\end{cases}
\end{equation}
%
%
In the cases except for $r=2$ or $r=3$, 
we list the parameters $r$ and $p$ which satisfy the condition of $(5.1)$. 
Following are the parameters $r\,(<50)$ and $p\,(<1000)$ :

\begin{itemize}
  
   \item $r=5$,\\
   $p=13, 17, 37, 53, 73, 97, 113, 137, 157, 173, 193, 197, 233, 257, 277,\\
   \qquad293,313, 317, 337, 353, 373, 397, 433, 457$, 
  
  \item $r=7$,\\
   $p=5, 13, 17, 41, 61, 73, 89, 97, 101, 157, 173, 181, 229, 241, 257,\\
   \qquad269, 293, 313, 349, 353, 397, 409, 433, 461$,
  
  \item $r=11$,\\
  $p=13, 17, 29, 41, 61, 73, 101, 109, 149, 173, 193, 197, 233, 241, 277, \\
  \qquad281, 293, 337, 349, 373, 409, 457, 461$,
  
  \item $r=13$,\\
  $p=5, 37, 41, 73, 89, 97, 109, 137, 149, 193, 197, 229, 241, 281, 293, \\
  \qquad317, 349, 353, 397, 401, 409, 421, 449, 457, 461$,
  
  \item $r=17$,\\
  $p=5, 29, 37, 41, 61, 73, 97, 109, 113, 173, 181, 193, 197, 233, 241, \\
  \qquad269, 277, 313, 317, 337, 397, 401, 449$,
  
  \item $r=19$,\\
  $p=13, 29, 37, 41, 53, 89, 97, 109, 113, 173, 181, 193, 241, 257, 269, \\
  \qquad281, 293, 317, 337, 373, 401, 409, 421, 433, 449$,
  
  \item $r=23$,\\
  $p=5, 17, 37, 53, 61, 89, 97, 109, 113, 137, 149, 157, 181, 229, 241, \\
  \qquad281, 293, 313, 337, 373, 389, 401, 421, 433, 457$,
  
  \item $r=29$,\\
  $p=17, 37, 41, 61, 73, 89, 97, 101, 113, 137, 157, 193, 229, 269, 293, \\
  \qquad317, 337, 389, 409, 421, 433, 449, 461$,
  
  \item $r=31$,\\
  $p=13, 17, 29, 37, 53, 61, 73, 89, 137, 181, 197, 229, 241, 269, 277, \\
  \qquad313, 337, 353, 389, 401, 409, 433, 449, 457, 461$,
  
  \item $r=37$,\\
  $p=5, 13, 17, 29, 61, 89, 97, 109, 113, 193, 241, 257, 277, 281, 313, \\
  \qquad353, 389, 401, 409, 421, 449, 457, 461$,
  
  \item $r=41$,\\
  $p=13, 17, 29, 53, 89, 97, 101, 109, 137, 149, 157, 181, 193, 229, 233, \\
  \qquad257, 281, 293, 313, 317, 397, 421, 457$,
  
  \item $r=43$,\\
   $p=5, 29, 37, 61, 73, 89, 113, 137, 149, 157, 233, 241, 257, 277, 313, \\
   \qquad349, 373, 389, 409, 421, 433, 449, 457$,
   
   \item $r=47$,\\
   $p=5, 13, 29, 41, 73, 109, 113, 137, 181, 193, 229, 233, 257, 281, 293, \\
   \qquad313, 317, 349, 373, 389, 409, 421, 433, 449, 461$.
  
\end{itemize}

    

\color{black}
\section*{Acknowledgments}
We are grateful to Professor Tsuyoshi Miezaki 
for helpful discussions on this research. 
\color{black}



\end{document}